\documentclass{amsart}
\usepackage{amsmath, amscd, amssymb, amsthm}
\usepackage{bbm}
\usepackage{latexsym}
\usepackage{amsfonts}
\usepackage{graphicx}
\usepackage[all,cmtip]{xy}

\usepackage[
dvipdfm, colorlinks=true, linkcolor=black,breaklinks=true,
urlcolor=blue, citecolor=black
]{hyperref}%
\usepackage{geometry}
\newtheorem{theorem}{Theorem}[section]
\newtheorem{lemma}{Lemma}

\newtheorem{question}{Question}

\newtheorem*{definition}{Definition}
\newtheorem{conjecture}{Conjecture}

\newcommand{\be}{\begin{equation}}
\newcommand{\ee}{\end{equation}}
\newcommand{\bea}{\begin{eqnarray}}
\newcommand{\eea}{\end{eqnarray}}

\newcommand{\vs}{\vspace{0.5cm}}
\newcommand{\vsss}{\vspace{0.35cm}}
\newcommand{\vsv}{\vspace{0.12cm}}

\def\XXint#1#2#3{{\setbox0=\hbox{$#1{#2#3}{\int}$ }
\vcenter{\hbox{$#2#3$ }}\kern-.6\wd0}}

\begin{document}

\title{On Curvature Tensors of Hermitian Manifolds}
\author{Bo Yang} \thanks{Research partially supported by an AMS-Simons Travel
Grant}

\address{Bo Yang. Department of Mathematics, Rutgers University,110 Frelinghuysen Road
Piscataway, NJ 08854, USA.} \email{{boyang@math.rutgers.edu} }

\curraddr{Bo Yang. Department of Mathematics, 310 Malott Hall,
Cornell University, Ithaca, NY 14853-4201, USA.}
\email{{boyang@math.cornell.edu} }

\author{Fangyang Zheng} \thanks{Research partially supported by NSFC 11271320}
\address{Fangyang Zheng. Center for Mathematical Sciences, Zhejiang University,
Hangzhou, 310027, Zhejiang, China and Department of Mathematics, The
Ohio State University, 231 West 18th Avenue, Columbus, OH 43210,
USA} \email{{zheng.31@osu.edu}}



\begin{abstract}
In this article, we examine the behavior of the Riemannian and
Hermitian curvature tensors of a Hermitian metric, when one of the
curvature tensors obeys all the symmetry conditions of the curvature
tensor of a K\"ahler metric. We will call such metrics
G-K\"ahler-like or K\"ahler-like, for lack of better terminologies.
Such metrics are always balanced when the manifold is compact, so in
a way they are more special than balanced metrics, which drew a lot
of attention in the study of non-K\"ahler Calabi-Yau manifolds. In
particular we derive various formulas on the difference between the
Riemannian and Hermitian curvature tensors in terms of the torsion
of the Hermitian connection. We believe that these formulas could
lead to further applications in the study of Hermitian geometry with
curvature assumptions.

\end{abstract}

\maketitle

\tableofcontents

\markleft{On Curvature Tensors of Hermitian Manifolds}
\markright{On Curvature Tensors of Hermitian Manifolds}

\section{Introduction}

In recent years, there has been much progress in the geometric
analysis of Hermitian manifolds, with the intent of pushing analysis
on K\"ahler manifolds to general Hermitian ones, and also with the
study of non-K\"ahler Calabi-Yau manifolds from string theory. See
for instance the work of Fu-Yau \cite{Fu-Yau}, Fu-Li-Yau
\cite{Fu-Li-Yau}, Fu \cite{Fu}, Fu-Wang-Wu \cite{Fu-Wang-Wu},
\cite{Fu-Wang-Wu1}, Liu-Yang \cite{Liu-Yang}, \cite{Liu-Yang1},
\cite{Liu-Yang2}, Streets-Tian \cite{StreetsTian}, Tosatti-Weinkove
\cite{Tosatti-Weinkove}, \cite{Tosatti}, Guan-Sun \cite{Guan-Sun},
and the references therein.

\vsv

Given a complex manifold $M^n$, a Hermitian metric $g$ is just a
Riemannian metric such that the almost complex structure $J$
preserves the metric. There are two canonical connections associated
with the metric, namely, the Hermitian (aka Chern) connection
$\nabla^h$ and the Riemannian (aka Levi-Civita) connection $\nabla
$. The first is the unique connection that is compatible with both
the metric and the complex structure, while the second is the only
torsion-free connection that is compatible with the metric. Let us
denote by $R^h$ and $R$ the curvature tensor of these two
connections.

\vsv

Here $R$ is just the Riemannian curvature tensor, and we extend it
linearly over ${\mathbb C}$. Both $R^h$ and $R$ are anti-symmetric
with respect to their first two or last two positions, and they are
both real operators. $R$ is also symmetric when its first two and
last two positions are interchanged, and satisfies the Bianchi
identity which means that when one positions is held fixed while the
other three are cyclicly permuted, the sum is always zero.

\vsv

Under the decomposition $TM\otimes {\mathbb C}=T^{1,0}M \oplus
T^{0,1}M$, all the components of $R^h$ vanish except
$R^h_{X\overline{Y}Z\overline{W}}$ (plus the obvious variation by
the skew-symmetries with respect to the first two or last two
positions), where $X$, $Y$, $Z$ and $W$ are type $(1,0)$ tangent
vectors. But in general it is not symmetric with respect to its
first and third (or its second and fourth) positions.

\vsv

When $g$ is K\"ahler, which means $\nabla^h=\nabla $, we have
$R^h=R$. In this case, the only non-trivial components are
$R_{X\overline{Y}Z\overline{W}}$, and
$R_{X\overline{Y}Z\overline{W}} = R_{Z\overline{Y}X\overline{W}}$.
For a general Hermitian metric, however, components
$R_{XYZ\overline{W}}$ and $R_{XY\overline{Z}\overline{W}}$ might be
non-zero, and $R$ may not be symmetric with respect to its first and
third (or second and fourth) positions. The only known condition is
that
$R_{XYZW}=R_{\overline{X}\overline{Y}\overline{Z}\overline{W}}=0$,
discovered by Gray in \cite{Gray} (see also formula (19) in \S 3).

\vsv

Of course when the metric $g$ is K\"ahler, one has  $R^h=R$. So a
naive question is, when a Hermitian metric $g$ satisfies $R^h=R$,
must it be K\"ahler? We could not seem to find an answer to this
question in the literature, to our surprise, so we took it to our
own hands and the first result in this article is simply to give a
positive answer to this question. That is, we have the following:

\begin{theorem} \label{result1} Given a Hermitian manifold $(M^n,g)$,
if its Riemannian curvature tensor $R$ and its Hermitian curvature
tensor $R^h$ are equal, then $g$ is K\"ahler.
\end{theorem}

Next, we would like to know what will happen when both $R$ and $R^h$
satisfy all the symmetry conditions of the curvature tensor of a
K\"ahler metric. To make things precise, let us first introduce the
following notion:


\begin{definition} [\textbf{K\"ahler-like and G-K\"ahler-like}]

A Hermitian metric $g$ will be called {\em K\"ahler-like,} if
$R^h_{X\overline{Y}Z\overline{W}}=R_{Z\overline{Y}X\overline{W}}^h$
holds for any type $(1,0)$ tangent vectors $X$, $Y$, $Z$, and $W$.
Similarly, if $R_{XY\overline{Z}\overline{W}}=R_{XYZ\overline{W}}=0$
for any type $(1,0)$ tangent vectors $X$, $Y$, $Z$, and $W$, we will
say that $g$ is {\em Gray-K\"ahler-like}, or  {\em G-K\"ahler-like}
for short.
\end{definition}


Note that the above definition simply means that $R^h$ (or $R$)
satisfies all the symmetry conditions obeyed by the curvature
tensors of  K\"ahler manifolds. The G-K\"ahler-like condition was
first introduced by Gray in \cite{Gray} (condition (1) on p. 605).
\vsv

When $g$ is K\"ahler-like, by
taking complex conjugations, we see that $R^h$ is also symmetric
with respect to its second and fourth positions, thus obeying all
the symmetries of the curvature tensor of a K\"ahler metric.

\vsv

Similarly, when $g$ is G-K\"ahler-like, we have $R_{XY\ast \ast }=0$
by the aforementioned Gray's Theorem, so the only non-trivial
components of $R$ are in the form $R_{X\overline{Y}Z\overline{W}}$.
Also,  the vanishing of $R_{XZ\overline{Y}\overline{W}}$ plus the
first Bianchi identity imply that $R_{Z\overline{Y}X\overline{W}} =
R_{X\overline{Y}Z\overline{W}}$. So $R$ obeys all the symmetries of
the curvature tensor of a K\"ahler metric.

\vsv

Of course being K\"ahler-like or G-K\"ahler-like does not mean that
the metric $g$ will have to be K\"ahler. There are plenty of
non-K\"ahler Hermitian metrics $g$ which are K\"ahler-like or
G-K\"ahler-like. For instance, when $n\geq 2$, there are Hermitian
manifolds that are non-K\"ahler but with $R^h=0$ everywhere. Such
manifolds are certainly K\"ahler-like. In \cite{Boothby}, Boothby
showed that a compact Hermitian manifold with $R^h=0$ everywhere is
the quotient of a complex Lie group by a discrete subgroup. For
$n\geq 3$, there are such manifolds that are non-K\"ahler.

\vsv

As we shall see in later sections, one can explicitly write down
Hermitian metrics in dimension $n\geq 2$ that are K\"ahler-like or G-K\"ahler-like but
non-K\"ahler. The first compact, non-K\"ahler example of
G-K\"ahler-like manifolds was observed by Gray in \cite{Gray}. It is
the {\em Calabi threefolds}, a family of compact complex manifolds
of dimension 3 with $c_1=0$ that are diffeomorphic to the product of
a compact Riemann surface with a real $4$-torus, discovered by
Calabi in 1958 \cite{Calabi}. For the reader's convenience, we give
a sketch of Calabi's construction in \S 3.

\vsv

Theorem \ref{result1} implies that any Hermitian manifold $(M^n, g)$
satisfying $R=R^h$ is K\"ahler. Note that under the assumption of
$R=R^h$, the manifold is obviously both K\"ahler-like and
G-K\"ahler-like in view of the above definition. In light of this,
we raise the following natural question:


\begin{conjecture}  \label{conjecture 1}
If a Hermitian manifold $(M^n,g)$ is both K\"ahler-like and
G-K\"ahler-like, then $g$ must be K\"ahler.
\end{conjecture}


At this point, we could not seem to establish a proof to this
conjecture in its full generality. However, we are able to prove a
partial result, which could be serving as a piece of supporting
evidence. To be precise,  we have the following:

\begin{theorem} \label{result2}
Let $(M^n,g)$ be a Hermitian manifold  that is both K\"ahler-like
and G-K\"ahler-like. If either $M^n$ is compact or $n\leq 3$,  then
$g$ is K\"ahler.
\end{theorem}

Note that K\"ahler-like or G-K\"ahler-like metrics provide important
classes of special Hermitian metrics. When the manifold is compact,
either condition would imply that the metric is balanced, namely,
$d(\omega^{n-1})=0$, where $\omega$ is the K\"ahler form. So each
type is more special than being balanced for compact manifolds.
More specifically, we have the following:

\begin{theorem} \label{balancedresult}
Let $(M^n,g)$ be a compact Hermitian manifold. If it is  either
K\"ahler-like or G-K\"ahler-like, then it must be balanced.
\footnote{An anonymous referee kindly brought to our attention that
Corollary 4.5 of \cite{Liu-Yang1} also implies that any compact
G-K\"ahler like manifold is balanced.}
\end{theorem}

In particular, on compact complex surfaces, K\"ahler-like or
G-K\"ahler-like metrics are  K\"ahler. In fact it was already
observed in \cite{Vaisman} that any compact G-K\"ahler-like
Hermitian surface is K\"ahler. Note that the completeness assumption
of the Hermitian metric plays an important role in view of the
example of noncompact G-K\"ahler-like and non-K\"{a}hler surface we
constructed in \S 3. In dimension $3$ or higher, there are examples
of compact non-K\"ahler manifolds that are K\"ahler-like or
G-K\"ahler-like, e.g., the Iwasawa threefold is Hermitian flat
(namely with vanishing $R^h$) thus K\"ahler-like; while the Calabi
threefolds are G-K\"ahler-like. Note that the Calabi threefolds also
have vanishing first Chern class. It would be a very interesting
question to classify all compact three dimensional non-K\"ahler
Hermitian manifolds that are K\"ahler-like or G-K\"ahler-like,
especially for Calabi-Yau threefolds (namely those with trivial
canonical bundle and finite fundamental group).

\vsv

In a larger context, recall that balanced metrics play an important
role in the Strominger system (\cite{Strominger}, \cite{Li-Yau},
\cite{Fu-Yau}). Mathematically, it is also intriguing to understand
the moduli space of Calabi-Yau threefolds (\cite{Reid} and
\cite{Tseng-Yau}). From Theorem \ref{balancedresult} we know that
K\"ahler-like or G-K\"ahler-like metrics on closed Hermitian
manifolds are more special than balanced ones. It might be
interesting to know if K\"ahler-like or G-K\"ahler-like metrics on
compact non-K\"ahlerian Calabi-Yau threefolds can play a role in the
study of Strominger system or the understanding of the moduli space
of Calabi-Yau threefolds.

\vsv

Next, let us consider the behavior of the K\"ahler-like or
G-K\"ahler-like condition under conformal changes. Since balanced
metrics are clearly unique (up to constant multiples) within each
conformal class, by Theorem \ref{balancedresult} we know that in the
compact case there can be at most one such metric within each
conformal class. In the non-compact case, one can write down the
equations and conclude that:

\begin{theorem} \label{conformalchange}
On a compact complex manifold $M^n$, each conformal class of
Hermitian metrics contains at most one metric (up to constant
multiples) that is K\"ahler-like or G-K\"ahler-like. For $(M^n,g)$
non-compact and  $\tilde{g}=e^{2u}g$ with $u\in C^{\infty }(M,
{\mathbb R})$,

1) if $g$ is K\"ahler-like, then $\tilde{g}$ is K\"ahler-like if
and only if $\partial \overline{\partial } u=0$;

2) if $g$ is G-K\"ahler-like, then $\tilde{g}$ is G-K\"ahler-like if and only if the function $\lambda =e^{-u}$ satisfies:
$ H_{\lambda} (X,Y)=0$ and $ \ \lambda H_{\lambda} (X,\overline{Y})= \langle X, \overline{Y}\rangle |\nabla \lambda |^2 $
for any type $(1,0)$ tangent vectors $X$ and $Y$. Here $ H_{\lambda}$ is the Hessian of $\lambda$.

In particular, $\lambda \Delta \lambda = n|\nabla \lambda |^2$ and$\
\Delta e^{(n-1)u} =0$.
\end{theorem}

In \cite{Liu-Yang}, \cite{Liu-Yang1}, and \cite{Liu-Yang2}, Liu and
Yang gave a detailed study of Hermitian manifolds and a thorough
analysis on the relationship of various Ricci tensors arising from
$R^h$ and $R$. Here we will introduce the right notion of Riemannian
bisectional curvature for $R$ and compare it with the Hermitian
bisectional curvature. The relationship between the two holomorphic
sectional curvatures is particularly simple, and obey a monotonicity
rule. See Theorem \ref{monotonicity} in \S 5 for more details.

Let us remark that an interesting aspect of our work is to derive
various formulas which characterize the difference between the
Riemannian and Hermitian curvature tensors in terms of the torsion
of Hermitian connection. We believe that these formulas could find
further applications in the study of Hermitian geometry with
curvature assumptions.

Next, we propose a natural question which should have been explored
before, but again we could not seem to find it in the literature. It
is well known that there are examples of non-K\"ahler Hermitian
manifolds with everywhere vanishing Hermitian curvature tensor
(i.e., with $R^h=0$ everywhere). In the compact case such manifolds
are all quotients of complex Lie groups, as proved in
\cite{Boothby}. Naturally one would wonder if one can classify
non-K\"ahler Hermitian manifolds with everywhere vanishing
Riemannian curvature tensor (i.e., with $R=0$ everywhere). We
propose the following

\begin{question} \label{riemannian flat}
Is there a characterization of Hermitian manifolds with vanishing
Riemannian curvature tensor? In the compact case, it amounts to
classify all compatible complex structures on the flat torus
$T^{2n}_{\mathbb R}$.
\end{question}

We will investigate Question \ref{riemannian flat} in a forthcoming
work.

The paper is organized as follows: In Section 2, we start from the
Cartan's structure equations and collect some preliminary results.
In Section 3, we discuss the Riemannian curvature tensor $R$ and the
Hermitian curvature tensor $R^h$ of a given Hermitian manifold, and
pay special attention to the cases when one or both of these
curvature tensors obey the symmetry conditions of the curvature
tensor of a K\"ahler manifold. In Section 4, we give proofs to
Theorems \ref{result1} and \ref{result2} stated in this section, and
in Section 5, we examine the uniqueness problem for such metrics
within a conformal class of Hermitian metrics. We also discuss the
notion of bisectional curvature for the Riemannian curvature tensor
$R$, and express the difference between the bisectional curvatures
in terms of a quadratic formula of the torsion tensor. In
particular, the holomorphic sectional curvatures of $R$ and $R^h$
obeys a simple monotonicity rule, with equality everywhere when and
only when the metric is K\"ahler.

\section{The structure equations of Hermitian manifolds}

\vsss

Let $(M^n,g)$ be a Hermitian manifold, with $n\geq 2$. We will
denote by $\nabla$ and $\nabla^h$ the Riemannian and Hermitian
connection of the metric $g$, and by $R$, $R^h$ their curvatures, called
the Riemannian or Hermitian curvature tensor, respectively.

\vsv

Let $A=\nabla - \nabla^h$ and denote by $T^h$ the torsion tensor of
$\nabla^h$:
$$ T^h(X,Y)=\nabla^h_XY - \nabla^h_YX - [X,Y] $$
for any two tangent vectors $X$, $Y$ on $M$. Since $\nabla$ is
torsion-free, the two tensors $A$ and $T^h$ are related by
\[
A_XY-A_YX = - T^h(X,Y).\]
So $T^h$ is the anti-symmetric part of
$A$. On the other hand, the compatibility of the connections with
the metric implies that
$$ \langle A_XY, Z\rangle + \langle A_YX,Z \rangle =
\langle X, T^h(Y,Z)\rangle + \langle Y, T^h(X,Z)\rangle $$
for any vector fields $X$, $Y$, and $Z$ on $M$. So $T^h$ completely
determines $A$. Here $\langle , \rangle $ is the (real) inner
product given by the Hermitian metric $g$.

\vsv

While the difference of $R$ and $R^h$ is given by $A$ and its first
covariant derivative, it seems to us that the torsion tensor $T^h$
would be easier to use in our context. Also, when the tangent frame
is chosen to be orthogonal (unitary), the dependence of $A$ on $T^h$
 takes the most convenient form. So we will use unitary coframes and
  focus on $T^h$ from now on.

\vsv

Let us complexify the tangent bundle and denote by $T^{1,0}M$ the
bundle of complex tangent vector fields of type $(1,0)$, namely,
complex vector fields in the form of $v- \sqrt{-1}Jv$, where $v$
is any real vector field on $M$.

\vsv

Suppose $\{ e_1, \ldots , e_n\}$ is a frame of $T^{1,0}M$ in a
neighborhood $M'\subseteq M$. Write $e=\ ^t\!(e_1, \ldots , e_n) $
 as a column vector. Denote by $\varphi = \ ^t\!(\varphi_1, \ldots ,
  \varphi_n)$ the column vector of $(1,0)$-forms in $M'$ which is the
coframe dual to $e$. For the Hermitian connection $\nabla^h$ of $g$,
let us denote by $\theta$,  $\Theta$ the matrices of connection and
curvature, respectively, and by $\tau$ the column vector of the
torsion $2$-forms, all under the local frame $e$. Then the structure
equations are
\begin{eqnarray}
d \varphi & = & - \ ^t\!\theta \wedge \varphi + \tau,  \label{formula 1}\\
d  \theta & = & \theta \wedge \theta + \Theta.
\end{eqnarray}
Taking exterior differentiation of the above equations, we get the
two Bianchi identities:
\begin{eqnarray}
d \tau & = & - \ ^t\!\theta \wedge \tau + \ ^t\!\Theta \wedge \varphi, \label{formula 3} \\
d  \Theta & = & \theta \wedge \Theta - \Theta \wedge \theta.
\end{eqnarray}
Note that under a frame change $\tilde{e}=Pe$, the corresponding
forms are changed by \[ \tilde{\varphi }= \ ^t\!P^{-1} \varphi , \ \
\ \  \tilde{\theta } = P\theta P^{-1} + dPP^{-1}, \ \ \ \
\tilde{\Theta }= P\Theta P^{-1}, \ \ \ \ \tilde{\tau } = \
^t\!P^{-1} \tau. \] In particular, the types of the $2$-forms in
$\Theta $ and $\tau$ are independent of the choice of the frame $e$.

\vsv

We will denote by $\langle , \rangle $ the (real) inner product
given by the Hermitian metric $g$, and by an abuse of notation, we
will again denote by $g$ the matrix $(\langle e_i,
\overline{e}_j\rangle)$ of the metric under the frame  $e$. The
compatibility of $\nabla^h$ with the metric implies \[ \theta g +
g\theta^{\ast } =dg, \ \  \ \ \Theta g + g\Theta^{\ast } =0, \]
where $\theta^{\ast } = \ ^t\!\overline{\theta}$. So when $e$ is a
unitary frame, $g=I$, and both $\theta $ and $\Theta $ are
skew-Hermitian. While when $e$ is holomorphic, $\theta$ is of type
$(1,0)$, thus $\tau $ must be of type $(2,0)$, and $\Theta$ cannot
have $(0,2)$ components, and its skew-Hermitian property for unitary
frames implies that it cannot have $(2,0)$ components, either. So
$\Theta$ must be of type $(1,1)$.

\vsv

In particular, when $e$ is holomorphic,  we have \[ \theta =
\partial g g^{-1}, \ \  \ \ \Theta =\overline{\partial }(\partial g
g^{-1}). \] We will write $ \omega = \sqrt{\!-\!1}\ ^t\!\varphi
\wedge g\overline{\varphi}$ and introduce the following
\begin{equation}
 \sigma = \ ^t\!\tau \wedge g \overline{\tau }.
\end{equation}
Both $\omega$ and $\sigma $ are independent of the choice of
the local frame, thus they are globally defined on $M$. $\omega $
 is the K\"ahler (aka fundamental or Hermitian or metric) form
of the Hermitian metric. It is everywhere positive definite. We will
call $\sigma$ the {\em torsion $(2,2)$-form}. It is a global,
nonnegative $(2,2)$ form on $M$, and $g$ is K\"ahler if and only if
$\sigma =0$ everywhere.

\vsss

Next, let us consider the Riemannian (aka Levi-Civita) connection
$\nabla $ of $g$. We will use $e$ and $\overline{e}$ as the frame
on the complexified tangent bundle $TM\otimes {\mathbb C}=T^{1,0}M
\oplus \overline{T^{1,0}M}$, so $ \varphi$ and $\overline{\varphi }$
 form the coframe. Write
\[ \nabla e = \theta_1 e + \overline{\theta_2 }\overline{e} , \ \ \ \
\ \nabla \overline{e} = \theta_2 e + \overline{\theta_1
}\overline{e}. \] Then the matrices of connection and curvature for
$\nabla $ becomes:
$$ \hat{\theta } = \left[ \begin{array}{ll} \theta_1 & \overline{\theta_2 } \\
\theta_2 & \overline{\theta_1 }  \end{array} \right] , \ \  \ \ \
\hat{\Theta } = \left[ \begin{array}{ll} \Theta_1 & \overline{\Theta}_2  \\
\Theta_2 & \overline{\Theta}_1   \end{array} \right], $$
 where
\begin{eqnarray}
\Theta_1 & = & d\theta_1 -\theta_1 \wedge \theta_1 -\overline{\theta_2} \wedge \theta_2, \\
\Theta_2 & = & d\theta_2 - \theta_2 \wedge \theta_1 - \overline{\theta_1 } \wedge \theta_2,  \label{formula 7}\\
d\varphi & = & - \ ^t\! \theta_1 \wedge \varphi - \ ^t\! \theta_2
\wedge \overline{\varphi }.
\end{eqnarray}
and under the frame change $\tilde{e}=Pe$, $ \overline{\tilde{e}} =
\overline{P}\overline{e}$, the above matrices of forms are changed
by \[ \tilde{\theta}_1  = P\theta_1P^{-1} +dPP^{-1}, \ \ \ \
\tilde{\theta}_2  = \overline{P} \theta_2 P^{-1}, \ \ \ \
\tilde{\Theta}_1 = P\Theta_1P^{-1}, \ \ \ \ \tilde{\Theta}_2 =
\overline{P}\Theta_2P^{-1}. \] We will write $\gamma = \theta_1 -
\theta $. Then $\tilde{\gamma } = P\gamma P^{-1}$ so $\gamma$
represents a tensor. The compatibility of $\nabla$ with the metric
implies
\begin{eqnarray*}
& & \theta_1 g + g\theta_1^{\ast } =dg, \ \ \ \ \
\theta_2 g + \overline{g} \ ^t\!\theta_2 =0, \\
& & \Theta_1 g + g\Theta_1^{\ast } =0, \ \ \ \ \Theta_2 g +
\overline{g} \ ^t\!\Theta_2 =0,
\end{eqnarray*}
where $g=(\langle e_i, \overline{e}_j\rangle )$  and $\alpha^{\ast
}=\ ^t\!\overline{\alpha }$ as before. So when $e$ is unitary, both
$\theta_2 $ and $\Theta_2$ are skew-symmetric, while $\theta_1$,
$\gamma$, and $\Theta_1$ are skew-Hermitian.

\vsv

Let us denote by $\gamma = \gamma ' + \gamma ''$ the decomposition
into $(1,0)$ and $(0,1)$ parts. Note that the following two
$2$-forms are independent of the choice of the frame $e$, thus are
globally defined on $M$:
\begin{equation}
\sigma_1 = - \sqrt{\!-\!1}\ \mbox{tr} (\gamma '\wedge \gamma'' ), \
\ \ \ \sigma_2 = \sqrt{\!-\!1} \ \mbox{tr} (\overline{\theta_2
}\wedge \theta_2 ). \ \ \
\end{equation}
We will see that both are nonnegative $(1,1)$ forms,
and vanish identically when and only when the metric is K\"ahler.


\begin{lemma}  \label{lemma 1}


{\em Each entry of $\theta_2$ is a $(1,0)$ form, and the $(0,2)$
component of $\Theta_2$ is zero.}
\end{lemma}

\begin{proof} Let $e$ be a local unitary frame. Write $\tau_i =
\sum_{j,k=1}^n T^i_{jk} \varphi_j\wedge \varphi_k$, where
$T^i_{jk}=-T^i_{kj}$. By (1) and (8), we get \[ ^t\! \gamma \wedge
\varphi + \tau + \ ^t\!\theta_2 \wedge \overline{\varphi } =0.\] Let
$\theta_2 =\theta_2 ' +\theta_2''$ be the decomposition into type
$(1,0)$ and $(0,1)$, respectively. The above equation gives
\begin{equation}
^t\!\theta_2'' \wedge \overline{\varphi } =0, \ \ \ ^t\! \gamma''
\wedge \varphi + ^t\!\theta_2' \wedge \overline{\varphi } =0, \ \ \
^t\! \gamma' \wedge \varphi + \tau =0.  \label{formula 10}
\end{equation}
Since $e$ is unitary, both $\theta_2 '$ and $\theta_2''$ are
skew-symmetric, and $\gamma'' = - \gamma'^{\ast }$. The first
equation in (\ref{formula 10}) implies that $\theta_2 ''=0$. Now by
(\ref{formula 7}), the $(0,2)$ part of $\Theta_2$ vanishes.

\end{proof}


\begin{lemma} \label{lemma 2}


Write $\tau_i = \sum_{j,k=1}^n T^i_{jk} \varphi_j\wedge \varphi_k$
with $T^i_{jk}=-T^i_{kj}$ under the frame $e$ and its dual coframe
$\varphi$. If $e$ is unitary, then
\begin{equation}
(\theta_2)_{ij} = \sum_{k=1}^n \overline{T^k_{ij}}\varphi_k, \ \ \
\gamma_{ij} = \sum_{k=1}^n (T^j_{ik}\varphi_k -
\overline{T^i_{jk}}\overline{\varphi}_k ).  \label{formula 11}
\end{equation}
\end{lemma}

\begin{proof} Under a unitary frame, $\gamma '' = - \
^t\!\overline{\gamma'}$. So by the last two equations in
(\ref{formula 10}) we get the coefficients of $\theta_2$ and
$\gamma'$ under the frame.
\end{proof}

\begin{lemma}  \label{lemma 3}

$\sigma_1$ and $\sigma_2$ are globally defined, nonnegative $(1,1)$
forms on $M$. The metric $g$ is K\"ahler
 if and only if any one of the following vanishes identically: $\tau$,
$\theta_2$, $\gamma '$, $\sigma$, $\sigma_1$, $\sigma_2$. Also,
$d\sigma_2 = \sqrt{\!-\!1} \mbox{tr} (\overline{\Theta}_2\theta_2 -
\overline{\theta_2}\Theta_2)$.
\end{lemma}

\begin{proof}
Under a frame change $\tilde{e}=Pe$, the matrices
$\overline{\theta}_2 \wedge \theta_2$ and $-\gamma'\wedge\gamma''$
are changed into $P\overline{\theta_2 }\wedge \theta_2P^{-1}$ and
$-P \gamma'\wedge\gamma''P^{-1}$, respectively, so their traces,
$\sigma_2$ and $\sigma_1$, are globally defined $(1,1)$ forms on
$M$. By (\ref{formula 11}), locally under any unitary frame $e$,
they can be expressed as
\begin{eqnarray}
\sigma_2 & = & \sqrt{\!-\!1}\sum_{k,l=1}^n (\sum_{i,j=1}^n
T^l_{ij}\overline{T^k_{ij}}) \varphi_k\wedge \overline{\varphi}_l, \\
\sigma_1 & = & \sqrt{\!-\!1} \sum_{k,l=1}^n (\sum_{i,j=1}^n T^j_{ik}
\overline{T^j_{il}}) \varphi_k\wedge \overline{\varphi}_l.
\end{eqnarray}
Therefore both are everywhere nonnegative, and the vanishing of
either of them is equivalent to the vanishing of $\tau$.  The
identity on $d\sigma_2$ is a direct consequence of (\ref{formula
7}).
\end{proof}

Next, let us recall the {\em torsion $1$-form} $\eta$ which is
defined to be the trace of $\gamma'$ (\cite{Gauduchon}). Under any
frame $e$, it has the expression:
\begin{equation}
\eta = \mbox{tr}(\gamma') = \sum_{i,j=1}^n T^i_{ij}\varphi_j.
\end{equation}
A direct computation shows that
\begin{equation}
\partial \omega^{n-1} = -2 \eta \wedge \omega^{n-1}.
\label{formula 15}
\end{equation}
Recall that the metric $g$ is said to be {\em balanced} if
$\omega^{n-1}$ is closed. The above identity shows that $g$ is balanced
if and only if $\eta =0$. When $n=2$, $\eta =0$ means $\tau =0$, so
balanced complex surfaces are K\"ahler. But for $n\geq 3$, $\eta $
contains less information than $\tau$.

\vsv

Let us conclude this section by pointing out the following fact,
which is probably well-known to experts in the field, but we give
the outline of proof here for readers' convenience.

\begin{lemma}  \label{lemma 4}
Given any point $p$ in a Hermitian manifold $(M^n,g)$, there exists
a unitary frame $e$ in a neighborhood of $p$ such that
$\theta|_p=0$.
\end{lemma}

\begin{proof} First we establish the following claim: Given any
$n\times n$ complex matrix $X$, there exists a $C^{\infty }$ map $f$
from a small
 disc $D$ in ${\mathbb C}$ into the unitary group $U(n)$ such that
 $f(0)=I$ and $\frac{\partial f}{\partial z}|_0=X$.

To prove the claim, let $P=X-X^{\ast }$, $Q=i(X+X^{\ast })$. Both are
skew-Hermitian, thus in the Lie algebra of $U(n)$. So there are
$1$-parameter subgroups $\phi$ and $\psi$ in $U(n)$ such that
$\phi'(0)=P$ and $\psi'(0)=Q$. Now let $f(z)=f(x+iy)=\phi(x)\psi(y)$.
We have $f(0)=I$, and
$\frac{\partial f}{\partial z}|_0 =
\frac{1}{2} (\frac{\partial f}{\partial x}-
i\frac{\partial f}{\partial y})|_0 =
\frac{1}{2} (\phi'(0)-i\psi'(0))=\frac{1}{2} (P-iQ)=X$.

Now by taking matrix products, we know there exists a smooth map $A$
from a small neighborhood of $p$ in $M^n$ into $U(n)$, such that
$\frac{\partial A}{\partial z_i}|_p=X_i$, $1\leq i\leq n$, for any
prescribed complex $n\times n$ matrices $X_1, \ldots , X_n$.

Take any unitary local frame $e$ near $p$. Write
$\theta|_p=\sum_{i=1}^n (-X_idz_i +X^{\ast}_i d\overline{z}_i)$.
Then $\tilde{e}=Ae$ will satisfy $ \tilde{\theta}|_p =(A\theta
A^{-1}+dAA^{-1})|_p=\theta|_p+dA|_p = 0$.

\end{proof}


\section{The Riemannian and Hermitian curvature tensors}

Now we turn our attention to the curvature tensors. Denote by
$R^h$, $R$ the curvature tensor of the Hermitian connection
$\nabla^h$ or the  Riemannian connection $\nabla$, respectively.
We have
\begin{eqnarray}
\Theta_{ij} & = & \sum_{k,l=1}^{n} R^h_{k\overline{l}i\overline{j}} \ \varphi_k \wedge \overline{\varphi_l}, \\
(\Theta_2)_{ij} & = & \sum_{k,l=1}^n (\frac{1}{2} R_{kl\overline{i}\overline{j}} \ \varphi_k \wedge \varphi_l + R_{k\overline{l}\overline{i}\overline{j}}  \ \varphi_k \wedge \overline{\varphi_l} ), \\
(\Theta_1)_{ij} & = & \sum_{k,l=1}^n ( \frac{1}{2}
R_{kli\overline{j}} \ \varphi_k\wedge \varphi_l +
R_{k\overline{l}i\overline{j}} \ \varphi_k \wedge
\overline{\varphi_l} +   \frac{1}{2}
R_{\overline{k}\overline{l}i\overline{j}} \ \overline{\varphi_k}
\wedge \overline{\varphi_l}  ).
\end{eqnarray}
Note that we have
\begin{equation}
R_{\overline{i}\overline{j}\overline{k}\overline{l}} = R_{ijkl} =0,
\end{equation}
because $\Theta_2^{0,2}=0$ by Lemma \ref{lemma 1}.
This property for general Hermitian metric was discovered by Gray in
\cite{Gray} (Theorem 3.1 on page 603), where it was stated as an
equation with $8$ real terms. (This perhaps once again illustrates
the usefulness of writing things in complex coordinates instead of
regarding $M$ as a real manifold with an integrable almost complex
structure $J$.)

\vsv

From (16), (17), (18), and the definition of K\"ahler-like and
G-K\"ahler-like in Section 1, it is easy to see that the following
hold:

\begin{lemma}  \label{lemma 5}
Given a Hermitian manifold $(M^n,g)$, $g$ is K\"ahler-like if and
only if $\ ^t\!\Theta \wedge \varphi=0$, and $g$ is G-K\"ahler-like
if and only if $\Theta_2=0$.
\end{lemma}

Note that the G-K\"ahler-like condition is equivalent to
$$ R_{xyuv} = R_{xyJuJv}$$
for any real tangent vectors $x$, $y$, $u$, $v$ on $M$. So this is
just the symmetry condition introduced by Gray in \cite{Gray}
(formula (1) on page 605).

\vsv

By Lemma \ref{lemma 5} and (\ref{formula 3}), $g$ being
K\"ahler-like would mean that under any frame $e$, we have $d\tau =
- \ ^t\!\theta \wedge \tau$. By the structure equation (\ref{formula
1})-(\ref{formula 3}), we know that under any unitary frame $e$, we
have
\begin{eqnarray}
\partial \omega = \sqrt{\!-\!1} \ ^t\!\tau \wedge \overline{\varphi },
\ \ \ \sqrt{\!-\!1}\partial \overline{\partial } \omega = \ ^t\!\tau
\wedge \overline{\tau} + \ ^t\!\varphi \wedge \Theta \wedge
\overline{\varphi }.
\end{eqnarray}
In particular, when $g$ is K\"ahler-like, we have
$\sqrt{\!-\!1}\partial \overline{\partial }  \ \omega = \sigma$.
In this case, if $M^n$ is compact and admits a positive
$(n\!-\!2,n\!-\!2)$ form $\chi$ that is $\partial \overline{\partial }$-closed,
then we can integrate $\sigma \wedge \chi$ and conclude that $\sigma$ must
be $0$, that is,



\begin{theorem} \label{thorem 6}

Let $(M^n,g)$ be a Hermitian manifold that is K\"ahler-like. If
$M^n$ is compact and admits a positive, $\partial \overline{\partial
}$-closed $(n\!-\!2,n\!-\!2)$ form $\chi$, then $g$ is K\"ahler. In
particular, if $M^n$ is compact,  K\"ahler-like, and $\partial
\overline{\partial }\ \omega^{n-2} =0$, then $g$ is K\"ahler. When
$n=2$, compactness implies that any K\"ahler-like metric is
K\"ahler.
\end{theorem}


In particular, if a compact complex threefold $M^3$ admits a
K\"ahler-like metric that is non-K\"ahler, then $M^3$ can not have
any pluriclosed metric (aka SKT metric, or Strongly K\"ahler with
Torsion).

\vsv

Boothby showed in \cite{Boothby} that any compact Hermitian manifold
$(M^n,g)$ with $R^h=0$ must be a quotient of a complex Lie group.
Of course any such manifold will be K\"ahler-like. One such example
is the famous Iwasawa manifold:

\vsss

\noindent {\bf Example (Iwasawa Manifold).} Consider the complex Lie
group $G$ formed by all complex $3\times 3$ matrices $X$ in the form
\[ X = \left[ \begin{array}{lll} 1 & x & z \\ 0 & 1 & y \\ 0 & 0 & 1
\end{array} \right]. \] Denote by $\Gamma $ the discrete subgroup of
$G$ of matrices with $x$, $y$, $z$ all in ${\mathbb Z} +
\sqrt{\!-\!1}{\mathbb Z}$. $\Gamma $ acts on $G$ by left
multiplication, leaving the holomorphic $1$-forms $dx$, $dy$, and
$dz-xdy$ invariant. So the three $1$-forms descend down to the
quotient $M^3=\Gamma \backslash G$ and form a global frame of the
cotangent bundle. Using these three $1$-forms to be the unitary
frame, we get a Hermitian metric on $M^3$ that is Hermitian flat,
namely, $R^h=0$.

\vsss

Note that the above manifold is non-K\"ahler but balanced. In general, it would be a very interesting problem to classify all compact $3$-dimensional complex manifolds that admit non-K\"ahler, K\"ahler-like metrics.

\vsv

Notice that when $(M^n,g)$ is K\"ahler-like, we have $d\tau = - \ ^t\!\theta \wedge \tau$, thus for a given point $p$ in $M$, if $e$ is a tangent frame such that $\theta|_p=0$, then $\overline{e_l}T^k_{ij}=0$. This implies that $\overline{\partial }\gamma ' =0$ at $p$. By taking trace, we get that $\overline{\partial }\eta =0$.

\begin{lemma} \label{lemma 6}
If a Hermitian manifold $(M^n,g)$ is K\"ahler-like, then its torsion
$1$-form $\eta$ is holomorphic. The converse of this is also true if
$n=2$.
\end{lemma}

Next let us consider the G-K\"ahler-like metrics, namely, those with
$\Theta_2=0$. Of course we are only interested in those that are
non-K\"ahler. The first G-K\"ahler-like but non-K\"ahler metric on
compact complex manifold was observed by Gray in \cite{Gray},
on Calabi threefolds discovered in \cite{Calabi}.

\vsss

\noindent {\bf Example (Calabi threefolds).} In 1958 Calabi
\cite{Calabi} discovered that $\mathfrak{X}=\mathfrak{X}^{\prime}
\times T^4$, with $\mathfrak{X}^{\prime}$ a hyperelliptic Riemann
surface with odd genus $g \geq 3$ and $T^4$ a real $4$-torus, can be
given a complex structure $J$ such that the resulting threefold
$(\mathfrak{X}, J)$ admits no K\"{a}hler metrics. Later Gray
\cite{Gray} showed that there exists a Hermitian metric which is
G-K\"{a}hler-like on $(\mathfrak{X},J)$.

\vsv

In more details, Calabi \cite{Calabi} proved that any orientable
hypersurface $M$ in $\mathbb{R}^7$ has a natural almost complex
structure $J$ induced from the space of purely imaginary octonions
which is isomorphic to $\mathbb{R}^7$, and $(M,g,J)$, with $g$ the
induced Riemannian metric from $\mathbb{R}^7$, has an almost
Hermitian structure. It was further proved in \cite{Calabi} (see
also Gray \cite{Gray1966}) that $(M,g,J)$ is Hermitian if and only
if $M$ is a minimal variety in $\mathbb{R}^7$. Based on these
results, Calabi began with a compact hyperelliptic Riemann surface
$\mathfrak{X}^{\prime}$, for example, the Riemann surface defined by
$\omega^2=\prod_{i=1}^{8} (z-z_i)$ where $z_i$ are distinct complex
numbers, and constructed three linearly independent Abelian
differentials which can be used to define an immersion $F_1$ from
$\widetilde{\mathfrak{X}}^{\prime}$, the universal Abelian covering
of $\mathfrak{X}^{\prime}$ locally into a minimal surface in
$\mathbb{R}^3$. Moreover, the covering transformations of
$F_1(\widetilde{\mathfrak{X}}^{\prime}) \subset \mathbb{R}^3$ are
given by translations in $\mathbb{R}^3$, By the results just
mentioned, the immersion $F:=F_1 \times Id:
\widetilde{\mathfrak{X}}^{\prime} \times \mathbb{R}^4 \rightarrow
\mathbb{R}^3 \times \mathbb{R}^4$ produces a Hermtian structure. It
can be proved that this complex structure $J$  is invariant under
translations in $\mathbb{R}^7$, Therefore, we can descend it to get
a compact Hermitian manifold $(\mathfrak{X},J)$ where
$\mathfrak{X}=\mathfrak{X}^{\prime} \times T^4$. Calabi proved that
$(\mathfrak{X}, J)$ does not admit any K\"ahler metrics.
Historically, Calabi threefolds was the first nontrivial example of
compact complex manifolds with zero first Chern class which are
diffeomorphic to K\"ahler manifolds but admit no K\"ahler metrics.

\vsv

Gray \cite{Gray1966} and \cite{Gray} further investigated the
curvature properties of such Hermitian manifolds. His result implies
that on $(\mathfrak{X}, g, J)$ with $g$ the the induced Riemannian
metric from $\widetilde{\mathfrak{X}}^{\prime} \times \mathbb{R}^4
\subset \mathbb{R}^7$, one has $ R_{xyuv} = R_{xyJuJv}$ for any real
tangent vectors $x$, $y$, $u$, $v$ on $\mathfrak{X}$. This means
that Calabi threefolds $(\mathfrak{X}, g, J)$ are G-K\"{a}hler-like.

\vsss

In the non-compact case, however, even in dimension $2$ there are
lots of such examples. For instance, we have the following:

\vsss

\noindent {\bf Example (G-K\"ahler-like surface).} Consider the
metric $g$ on ${\mathbb C}\times {\mathbb H}$ given by \[ \omega_g =
i(-iz_2+i\overline{z}_2)^2dz_1 \wedge d\overline{z}_1 +idz_2\wedge
d\overline{z}_2,\] where ${\mathbb H}$ is the upper half plane.
Write $(-iz_2+i\overline{z}_2)^2=\lambda =e^{2u}>0$ on ${\mathbb
H}$. Then under the natural frame of $\{ z_1, z_2\}$, we have
\begin{eqnarray*}
\theta_1 &= & \left[ \begin{array}{cc} du & -\lambda \overline{\mu} \\
 \mu & 0 \end{array} \right] , \ \ \ \ \ \ \
\theta_2 \ = \ \left[ \begin{array}{cc} 0 &  -\lambda \nu \\
\nu & 0 \end{array} \right], \\
\Theta_2 &  = &  d\theta_2 -\overline{\theta}_1 \wedge \theta_2 -
\theta_2 \wedge \theta_1 \ =
\ \left[ \begin{array}{cc} 0 &  -d(\lambda \nu )-\lambda \nu du \\
d\nu -\nu du & 0 \end{array} \right],
\end{eqnarray*}
where $\mu = u_2dz_1$, $\nu = u_{\overline{2}}dz_1$. Since
$u=\ln (-iz_2+i\overline{z}_2)$, $u_{\overline{2}2} =
- u_{\overline{2}}u_2$, $u_{\overline{2}\overline{2}} =
 - (u_{\overline{2}})^2$, and $d\lambda =2\lambda du$,
 we get $\Theta_2=0$, so the metric is G-K\"ahler-like.

\vsss

\noindent {\bf Remark:} For a non-compact complex manifold $M^n$, if
$g_0$ is a K\"ahler metric on $M$ and $f$ is a holomorphic function
$M$ that is nowhere zero, then $g=|f|^2g_0$ is K\"ahler-like, and is
non-K\"ahler if $n\geq 2$ and $f$ is not a constant.

\vsss

Next let us compute the curvatures $R$ and $R^h$ in terms of the
torsion components $T^i_{jk}$ and their derivatives. We have the
following:

\begin{lemma}  \label{lemma 7}
Let $(M^n,g)$ be a Hermitian manifold. Under any local unitary frame $e$, we have
\begin{eqnarray}
2T^k_{ij,\ \overline{l}} & = & R^h_{j\overline{l}i\overline{k}} -
R^h_{i\overline{l}j\overline{k}},
\label{formula 21}\\
R_{ijk\overline{l}} \ & = & T^l_{ij,k} + T^l_{ri} T^r_{jk} -
T^l_{rj} T^r_{ik},  \label{formula 22}
\end{eqnarray}
\begin{eqnarray}
R_{ij\overline{k}\overline{l}} & = &
T^l_{ij,\overline{k}} - T^k_{ij,\overline{l}} +
2T^r_{ij} \overline{T^r_{kl}} + T^k_{ri} \overline{ T^j_{rl} } +
T^l_{rj} \overline{ T^i_{rk} } - T^l_{ri} \overline{ T^j_{rk} } -
T^k_{rj} \overline{ T^i_{rl} },   \label{formula 23}\\
R_{k\overline{l}i\overline{j}} & = &
R^h_{k\overline{l}i\overline{j}} - T^j_{ik,\overline{l}} -
\overline{ T^i_{jl,\overline{k}} } + T^r_{ik} \overline{ T^r_{jl} }
- T^j_{rk} \overline{ T^i_{rl} } - T^l_{ri} \overline{ T^k_{rj} },
\label{formula 24}
\end{eqnarray}
where the index  $r$ is summed over $1$ through $n$, and the indices after the comma
denote the covariant derivatives with respect to the Hermitian connection $\nabla^h$.
\end{lemma}

\begin{proof} Using the structure equations, the Bianchi identities, and Lemma \ref{lemma 2},  we get the above identities by a straight
forward computation.
\end{proof}

\vsv

As  an immediate consequence of Lemma \ref{lemma 7}, we have the
following:

\begin{lemma} \label{lemma 8}

Let $(M^n,g)$ be a Hermitian manifold. If $g$ is G-K\"ahler-like,
then
\begin{equation*}
\overline{\partial }\eta \wedge \omega^{n-1} =- \eta \wedge
\overline{\eta} \wedge \omega^{n-1}.
\end{equation*}

\end{lemma}

\begin{proof}
Let us fix any point $p$ in $M^n$ and let $e$ be a
unitary frame in a neighborhood of $p$ such that $\theta|_p=0$.
Since $\Theta_2=0$, by taking $k=i$, $l=j$ in (23) of Lemma 7 and
sum them over, we get
$$  \sum_{i=1}^n \eta_{i,\overline{i}} = \sum_{r=1}^n |\eta_r|^2 ,$$
so Lemma 8 is proved.

\end{proof}


Now we are ready to prove Theorem \ref{balancedresult}.


\begin{proof}
[\textbf{Proof of Theorem \ref{balancedresult}}] Suppose $(M^n,g)$
is a compact Hermitian manifold with $n\geq 2$. By (\ref{formula
15}), we get
\begin{eqnarray*}
\overline{\partial }\partial \omega^{n-1} &= &
\overline{\partial }(-2\eta \wedge \omega^{n-1}) \\
& = & -2 \overline{\partial }\eta \wedge \omega^{n-1} - 4  \eta
\wedge \overline{\eta} \wedge \omega^{n-1}.
\end{eqnarray*}
If $g$ is G-K\"ahler-like, then the above calculation leads to
\[
\overline{\partial }\partial \omega^{n-1} = - 2 \eta \wedge
\overline{\eta} \wedge \omega^{n-1}  = - 2 ||\eta||^2 \omega^n.
\]
Now since $M$ is compact, integrating over $M$ would yield $\eta
=0$.

If instead $g$ is K\"ahler-like, then $\overline{\partial }\eta =0$
by Lemma \ref{lemma 6}, and the above argument again yields $\eta
=0$. So in either case $(M^n,g)$ is balanced.

\end{proof}


\section{The proofs of Theorems \ref{result1} and \ref{result2}}

In this section, we will give proofs to Theorems \ref{result1}
 and \ref{result2} stated in the introduction. First
let us assume that $(M^n,g)$ is a Hermitian manifold that is both
K\"ahler-like and G-K\"ahler-like. Fix any point $p\in M$, and let
$e$ be a unitary frame near $p$ such that $\theta|_p=0$. Since $g$
is K\"ahler-like, we have $R^h_{i\overline{j} k\overline{l}} =
R^h_{k\overline{j}i\overline{l}}$, thus at the point $p$ it holds
$$ T^k_{ij, \ \overline{l}} =0.$$
Therefore, by formula (\ref{formula 23}) in Lemma \ref{lemma 7}, we
know that

\begin{lemma}  \label{lemma 9}
If a Hermitian manifold $(M^n,g)$ is both K\"ahler-like and
G-K\"ahler-like, then under any unitary frame the following identity
\begin{equation}
2 \sum_{r=1}^n T^r_{ij} \overline{T^r_{kl}} \ = \  \sum_{r=1}^n \{
T^l_{ri} \overline{ T^j_{rk} } + T^k_{rj} \overline{ T^i_{rl} } -
T^k_{ri} \overline{ T^j_{rl} } -  T^l_{rj} \overline{ T^i_{rk} } \}
\label{formula 25}
\end{equation}
holds for any indices $i$, $j$, $k$, $l$. In particular, $M$ must be
balanced.
\end{lemma}

\begin{proof}

We are only left to prove the last statement, namely, the torsion
$1$-form $\eta$ is zero. By taking $k=i$, $l=j$ in (\ref{formula
25}) and sum over $i$ and $j$, we get
$$ 2||T||^2 = 2||T||^2 - 2 ||\eta ||^2 ,$$
where $||T||^2 = \sum_{i,j,k=1}^n |T^k_{ij}|^2$ and $||\eta ||^2=
\sum_{k=1}^n |\sum_{i=1}^n T^i_{ik}|^2$. Hence $\eta =0$.
\end{proof}


\begin{proof}
[\textbf{Proof of Theorem \ref{result1}}] Let $(M^n,g)$ be a
Hermitian manifold such that the Riemannian curvature tensor $R$ and
the Hermitian curvature tensor $R^h$ are equal to each other. Then
both $R$ and $R^h$ satisfy all symmetry conditions of the curvature
tensor of a K\"ahler manifold, so $g$ will be both K\"ahler-like and
G-K\"ahler-like. Thus by Lemma \ref{lemma 9}, we know that the
formula (\ref{formula 25}) holds. Fix any $p$ in $M$ and choose a
unitary frame $e$ such that $\theta |_p=0$. Since $R=R^h$, the
formula (\ref{formula 24}) in Lemma \ref{lemma 7} gives
\begin{equation}
\sum_{r=1}^n T^r_{ik} \overline{ T^r_{jl} }  \ = \ \sum_{r=1}^n  \{
T^j_{rk} \overline{ T^i_{rl} } + T^l_{ri} \overline{ T^k_{rj} } \}
\label{formula 26}
\end{equation}
for any indices $i$, $j$, $k$, $l$. By letting $j=i$, $l=k$ in
(\ref{formula 26}), we get \[ \sum_{r=1}^n |T^r_{ik}|^2 =
\sum_{r=1}^n \{ |T^i_{rk}|^2 + |T^k_{ri}|^2 \}. \] If we sum over
$i$ and $k$, it leads to $||T||^2=2||T||^2$, hence $T=0$ and $g$ is
K\"ahler. This completes the proof of Theorem \ref{result1}.
\end{proof}


\begin{proof} [\textbf{Proof of Theorem \ref{result2}}]
Now let us consider a Hermitian manifold $(M^n,g)$ which is both
K\"ahler-like and G-K\"ahler-like. By Lemma \ref{lemma 9}, we know
that (\ref{formula 25}) holds and $g$ is balanced. Letting $k=i$ and
$l=j$ in (\ref{formula 25}), we get
\begin{equation}
2 \sum_{r=1}^n |T^r_{ij}|^2  \ = \ \sum_{r=1}^n  \{  |T^i_{rj}|^2  +
| T^j_{ri}|^2 - 2Re ( \overline{ T^i_{ri} } T^j_{rj} ) \}.
\label{formula 27}
\end{equation}
Also, by formula (\ref{formula 22}) in Lemma (\ref{lemma 7}), we get
$T^l_{ij,k} = \sum_{r=1}^n \{ -  T^l_{ri} T^r_{jk} + T^l_{rj}
T^r_{ik}\} $ for any indices $i$, $j$, $k$, and $l$. Letting $l=j=s$
and sum over $s$, and using the fact that the metric is a balanced
one, we get
\begin{equation}
\sum_{r,s=1}^n T^s_{ri} T^r_{sk} \ = \ 0  \label{formula 28}
\end{equation}
for any indices $i$ and $k$.

Now we are ready to prove Theorem \ref{result2}. When $n=2$,
balanced metrics are K\"ahler, so $g$ is already K\"ahler. Now
assume that $n=3$. Let $(ijk)$ be any cyclic permutation of $(123)$.
Write $a_i=T^i_{jk}$, $b_i=T^j_{ij}=-T^k_{ik}$. The last equality
holds true because of the fact that $\eta =0$.

Since $T^k_{ij}=-T^k_{ji}$, the identities (\ref{formula 27}) and
(\ref{formula 28}) lead to
\begin{eqnarray}
& & |a_i|^2 + |a_j|^2 - 2|a_k|^2  \ = \  |b_i|^2 + |b_j|^2 -
2|b_k|^2,
\label{formula 29}\\
& & b_i b_j \ = \ b_k a_k,  \label{formula 30}\\
& & a_i a_j \ = \ b_k^2   \label{formula 31}
\end{eqnarray}
whenever $(ijk)$ is a cyclic permutation of $(123)$. Note that
(\ref{formula 31}) is obtained by letting $i=k$ in (\ref{formula
28}) first.

If one of the $b_i$ is zero, say, $b_1=0$, then by (\ref{formula
30}), $b_2b_3=0$. Without loss of generality, let us assume that
$b_2=0$. If $b_3\neq 0$, then by (30) $b_1b_2=b_3a_3$ so $a_3=0$,
and by (31), $a_1a_2=b_3^2\neq 0$. But then (29) gives
$|a_1|^2+|a_2|^2 = -2 |b_3|^2$, a contradiction. So we must have
$b_3=0$ as well. In this case, (\ref{formula 31}) implies that at
least two of the $a_i$'s must be zero, while (\ref{formula 29})
implies that the third one is also zero. So all $a_i$ and $b_i$ are
zero, that is, $T=0$, thus $g$ is K\"ahler.

Now assume that $b_1b_2b_3\neq 0$. Then by (\ref{formula 30}), we
have $a_k=\frac{b_ib_j}{b_k}$. By letting $l=i$ in formula
(\ref{formula 25}) in Lemma \ref{lemma 9}, we get through a direct
computation that \[ b_j \overline{b_k} + b_i\overline{a_j} + a_k
\overline{b_i} =0. \] Plugging in $a_j=(b_ib_k)/b_j$ and
$a_k=(b_ib_j)/b_k$, we get \[ b_j\overline{b_k} (1+
\frac{|b_i|^2}{|b_j|^2} + \frac{|b_i|^2}{|b_k|^2}) =0,\]
contradicting to the assumption that $b_i$'s are non-zero. This
completes the proof of Theorem \ref{result2} for the case when
$n=3$.

\vsv

Now let us assume that $n\geq 2$ is arbitrary but $M^n$ is compact.
Let $e$ be a local unitary frame. For a tensor $P^k_{ij}$ of type
(1,2), we will denote by
$$P^k_{ij,l}  \ \  \mbox{and} \ \  P^k_{ij,\overline{l}}$$ the covariant
differentiation of $P$ with respect to the Hermitian connection
$\nabla^h$. So if $e$ is a frame such that $\theta|_p=0$ at the
fixed point $p$, then at the point $p$ we have $
P^k_{ij,l} = e_l( P^k_{ij} )$ and $
P^k_{ij,\overline{l} }= \overline{e}_l ( P^k_{ij} ) $. In
particular, by the fact that $g$ is both K\"ahler-like and
G-K\"ahler-like, we get formula (\ref{formula 25}) and its special
case (\ref{formula 27}), as well as the following
\begin{eqnarray}
T^k_{ij,\overline{l} } & = & 0,   \label{formula 32} \\
T^k_{ij,l} & = & \sum_{r=1}^n \{ -  T^k_{ri} T^r_{jl} + T^k_{rj}
T^r_{il}\}  \label{formula 33}
\end{eqnarray}
for any indices $i$, $j$, $k$, and $l$.  Now we use the assumption
that $M$ is compact. Note that if $f$ is any smooth function on
$M^n$ such that $Lf:=\sum_{l=1}^n f_{,l\overline{l}}\geq 0$
everywhere, then by Bochner's Lemma (see \cite{Boothby}), $Lf=0$
everywhere, and $f$ is a constant. Consider the function
$f=\sum_{i,j,k=1}^n |T^k_{ij}|^2$ under any unitary frame $e$. Then
we have
\begin{eqnarray*}
L f & = & ( T^k_{ij,l} \overline{T^k_{ij}}  )_{,\overline{l}} \ =
\  |T^k_{ij,l}|^2 + (T^k_{ij,l})_{,\overline{l}} \ \overline{T^k_{ij}} \\
& = & |T^k_{ij,l}|^2 + ( -  T^k_{ri} T^r_{jl} + T^k_{rj}
T^r_{il})_{,\overline{l}} \overline{T^k_{ij}} \ = \  |T^k_{ij,l}|^2,
\end{eqnarray*}
where the third equality above is because of (\ref{formula 33}),
while the others are because of (\ref{formula 32}). So we have
$T^k_{ij,l}=0$ and
\begin{equation}
\sum_{r=1}^n   T^k_{ri} T^r_{jl}  = \sum_{r=1}^n  T^k_{rj} T^r_{il}
\label{formula 34}
\end{equation}
for any indices $i$, $j$, $k$, and $l$.

Write $V=T^{1,0}_p M$ for a given point $p \in M$, and then for any
$X=\sum_i X_ie_i\in V$, let us denote by $A_X$ the $n\times n$
matrix $(\sum_i X_iT_{ij}^k)_{j,k=1}^n$, which represents a linear
transformation from $V$ into itself. By multiplying $X_iX_l$ onto
(\ref{formula 34}) and adding up $i$ and $l$, we get $(A_X)^2=0$.
Also, for $Y=\sum_i Y_ie_i $ in $V$, if we respectively multiplying
$X_iY_l$ or $X_lY_i$ onto (\ref{formula 34}) and adding up $i$ and
$l$, we get $A_XA_Y=-A_YA_X$.

\vsv

{\bf Claim 1:} There exists $W\in V$ such that $A_Y(W)=0$ for any
$Y$ in $V$.

\vsv

To see this, it suffices to prove the following slightly more
general statement about anti-commutative system of step-$2$
nilpotent matrices:

\vsv

{\bf Claim 2:} For any given integer $m$, Let $\{ A_1 , \ldots ,
A_m\}$ be a set of $n\times n$ complex matrices satisfying the
condition
\begin{equation}
A_iA_j=-A_jA_i, \   \ \ \ \forall \ 1\leq i, j\leq m.
\end{equation}
Then $\bigcap_{i=1}^m N(A_i) \neq 0$, where $N(A_i)$ denotes the kernel of $A_i$.

\vsv

We will use induction on $n$ to prove Claim 2. We may assume that
these $A_i$ are linearly independent, as otherwise we could just
reduce the number $m$. When $n=2$, since $A_1^2=0$, there exists
non-singular $2\times 2$ matrix $P$ such that $PA_1P^{-1}=E$, where
\begin{equation*}
E=\left[ \begin{array}{ll} 0 & 1 \\ 0 & 0 \end{array} \right]
\end{equation*}
For any $i\geq 2$, since $PA_iP^{-1}$ is nilpotent and anti-commutative with $E$, it must be in the form $a_iE$ for some constant $a_i$. So all these $A_i$ have common kernel.

\vsv

For general $n$, let us assume that $A_1$ has the largest rank among all linear combinations of these $A_i$. By a base change, we know that there exists a non-singular matrix $P$ such that $PA_1P^{-1}$ takes the form
\begin{equation*}
PA_1P^{-1} = \left[ \begin{array}{ccc} 0 & 0 & I_k \\ 0 & 0 & 0 \\ 0 & 0 & 0 \end{array} \right]
\end{equation*}
where $I_k$ is the identity matrix and $k=\mbox{rank}(A_1)$ where
$2k \leq n$. For any $i\geq 2$, since the rank of $\lambda A_1 +
A_i$ is at most $k$ for any $\lambda \in {\mathbb C}$, we know that
the lower left corner of $PA_iP^{-1}$ must be zero:
\begin{equation*}
PA_iP^{-1} = \left[ \begin{array}{ccc} B_i & \ast & \ast\\ 0 & 0 &
\ast  \\ 0 & 0 & -B_i \end{array} \right].
\end{equation*}
(The lower right corner is $-B_i$ because $A_1A_i=-A_iA_1$.) Note
that these $k\times k$ matrices $\{ B_i\}$ also satisfy $(35)$, so
by induction on $n$, we know that these $B_i$, thus all the $A_i$,
will have a common kernel. This proves Claim 2, hence Claim 1.

There is an alternative proof of Claim 2, which is constructive in
nature and might be interesting in its own right. \footnote{The
authors are indebted to an anonymous referee for suggesting the
alternative proof.} The proof goes as follows:

Let us suppose that dim $V = n$ and rank$A_X = k > 0$, then $n =
2k+l$, with $l \neq 0$. In this situation we can find a basis $\{
v_1, \ldots ,  v_k, x_1, \ldots ,  x_l,  y_1, \ldots , y_k\} $  of
$V$ such that $\{ v_1, \ldots ,  v_k\}$  is a basis for $V_1 =
\mbox{Im}A_X$, $x_1=X$,  and such that $A_Xy_i = v_i$ and $\{ v_1,
\ldots , v_k, x_1, \ldots , x_l\}$  is a basis for ker$A_X$. Note
that in our situation we have $(A_Y )^2 = 0$, $A_Y A_Z = - A_ZA_Y$,
and $A_Y Z = - A_ZY$. Of course, $A_{\mu Y +\nu Z} = \mu A_Y + \nu
A_Z$. We take  $W_1 = A_{y_1}A_{y_2} \cdots A_{y_{k-1}} (v_k)$.
Because $A_{y_i}v_j = A_{y_i}A_Xy_j = A_XA_{y_j} y_i = -A_{y_j}
v_i$,  $i \neq j$, note that
$$ W_1 = (-1)^{k-i} A_{y_1} \cdots A_{y_{i-1}} A_{y_{i+1}}  \cdots A_{y_k} (v_i) $$
Also note that $A_{v_i}v_j = A_{v_i} A_Xy_j = -A_{v_i}A_{y_j}X =
A_{y_j}A_{v_i}X = -A_{y_j}A_Xv_i = 0$. Moreover, $A_{x_i}v_j =
A_{x_i}A_Xy_j = -A_{y_j} A_Xv_i = 0$. Therefore $A_{v_i}W_1 = 0$ and
$A_{x_i}W_1 = 0$. Another remark is that $A_{y_i}v_i = 0$.
Therefore, $A_{y_i}W_1 = 0$. In summary, $A_ZW_1 = 0$, for all $Z
\in  V$. If $W_1 \neq 0$, we would have proved Claim 2.

\vsv

If $W_1 = 0$, we choose a set of $k - 1$ indices. For instance, we consider the
indices $1, \ldots , k-2, k-1$ and take the element
$W_2 = A_{y_1}A_{y_2} \cdots A_{y_{k-2}} (v_{k-1})$.
We already have $A_{y_k}W_2 = 0$, because $W_1 = 0$. Moreover, as before,
$A_{v_i}W_2 = 0$, $A_{x_j}W_2 = 0$ and $A_{y_i}W_2 = 0$, for $i = 1;, \ldots ,  k$ and $j = 1, \ldots , l$.
If $W_2 \neq 0$ for some set of $k - 1$ indices, we would have proved the claim.
If $W_2 = 0$ for all set of $k-1$ indices, we would choose a set of $k-2$ indices,
etc. For some set of $k - j$ indices, we have to obtain $W_{j+1} \neq  0$. In the worst
situation, we would have done $k$ steps and $W_{k+1} = v_1 \neq 0$. Because of the
previous steps $A_Zv_1 = 0$, for all $Z \in  V$ and the claim would be proved.

\vsv

Now that Claim 1 is established, there exists non-zero tangent
vector $X$ in $T_{p}^{1,0}M$ such that $T^k_{Xj}=0$ for any $j$,
$k$. Let us choose unitary frame $e$ so that $X$ is parallel to
$e_n$. Then we have $T^k_{nj}=0$ for any $j$, $k$. Let $i=n$ in
(27), we get
$$ \sum_{r=1}^n |T^n_{rj}|^2 =0 $$
for any $j$. Therefore $T^n_{jk}=0$ for any $j$, $k$. That is, the
components of the torsion tensor $T^k_{ij}=0$ whenever any of the
indices is $n$. Repeating this argument, we conclude that
$T^k_{ij}=0$ whenever any of the indices is greater than $2$. Then
$T$ must be $0$ since $g$ is balanced, and this completes the proof
that $g$ must be K\"ahler.
\end{proof}

\section{The conformal change of metrics and bisectional curvatures}

Let $(M^n,g)$ be a Hermitian manifold, $u\in C^{\infty }(M)$ a
real-valued smooth function, and $\tilde{g}=e^{2u}g$ a conformal
change of the metric.

\vsv

Let $e$ be (the column vector of) a local unitary frame of $g$, with
(the column vector of) the dual coframe $\varphi$. Then
$\tilde{e}=e^{-u}e$ and $\tilde{\varphi }=e^u\varphi$ are local
unitary frame and coframe with respect to the metric $\tilde{g}$.

\vsv

Denote by $\tilde{\theta}$ and $\tilde{\Theta}$ the matrix of
Hermitian connection and Hermitian curvature of the metric
$\tilde{g}$ with respect to the unitary frame $\tilde{e}$, then it is
easy to see that
$$ \tilde{\theta } = \theta + (\partial u-\overline{\partial }u)I,
\ \ \ \ \tilde{\Theta } = \Theta -2 \partial \overline{\partial }u
I, \ \ \ \ \ $$ where $\theta$ and $\Theta$ are the matrix of
Hermitian connection and Hermitian curvature of $g$ under $e$. From
that, we get \[ \tilde{\tau} = e^u (\tau + 2\partial u \wedge
\varphi )  \ \ \ \ \mbox{and}\]
\begin{equation}
e^u \tilde{T}^i_{jk} = T^i_{jk} + u_j\delta_{ik} - u_k \delta_{ij}
\end{equation}
where $u_j=e_j(u)$. Using Lemma \ref{lemma 2}, we get the following:


\begin{lemma}  \label{lemma 11}
Let $e$, $\tilde{e}=e^{-u}e$ be the local unitary frames for $g$ and
$\tilde{g}=e^{2u}g$, respectively. Then the connection matrixes are
related as
\begin{eqnarray}
\tilde{\theta}_1 & = & \theta_1 + v\  ^t\!\varphi -
\overline{\varphi }\ v^{\ast }, \\
\tilde{\theta}_2 & = & \theta_2  + \overline{v} \  ^t\!\varphi  -
\varphi \ v^{\ast }
\end{eqnarray}
where $v=\ ^t\!(u_1, \ldots , u_n)$.
\end{lemma}


Now we are ready to prove Theorem \ref{conformalchange}.


\begin{proof}
[\textbf{Proof of Theorem \ref{conformalchange}}] When $M^n$ is a
compact complex manifold, by Theorem \ref{balancedresult},
K\"ahler-like or G-K\"ahler-like metrics are balanced, and balanced
metrics are clearly unique (up to constant multiples) within each
conformal class, so each conformal class of Hermitian metrics on
$M^n$ can contain at most one (up to constant multiples)
K\"ahler-like or G-K\"ahler-like metric.

\vsv

Now assume that $(M^n,g)$ is a non-compact Hermitian manifold. Let
$u$ be a real-valued smooth function on $M^n$ and
$\tilde{g}=e^{2u}g$ be a conformal change of $g$. As in the above,
let $e$ be a local unitary frame of $g$ with dual coframe $\varphi$,
then $\tilde{e}=e^{-u}e$ and $\tilde{\varphi }=e^u\varphi $ are
unitary frame and coframe for $\tilde{g}$. We have $\tilde{\Theta
}=\Theta - 2 \partial \overline{\partial } uI$. So when $g$ is
K\"ahler-like, which means that $\ ^t\Theta \wedge \varphi =0$,
the metric $\tilde{g}$ will be K\"ahler-like if and only $\
^t\tilde{\Theta }\wedge \tilde{\varphi }=0$, which is equivalent to
$\ \partial \overline{\partial } u =0$.

\vsv

Next let us assume that $\tilde{\Theta}_2 - \Theta_2=0$. By Lemma
\ref{lemma 11} and a somewhat lengthy but straight forward
computation, we get the following equations for $\lambda = e^{-u}$:
\begin{eqnarray*}
& & e_i(\lambda_j) - \theta_{jk}(e_i)\lambda_k + T^k_{ij}\lambda_k =0, \\
& & \overline{e}_i(\lambda_j) - \theta_{jk} (\overline{e}_i)\lambda_k -
\overline{T}^j_{ik}\lambda_k - T^i_{jk} \overline{\lambda}_k \  =
\  2\delta_{ij} |\lambda_k|^2 /\lambda
\end{eqnarray*}
for any $i$, $j$. Note that the index $k$ is summed up in the above
identities. Let $H_{\lambda }$ be the Hessian of the function
$\lambda$ with respect to the Riemannian metric $g$, then the above
equations are simply saying that $$ H_{\lambda}(X,Y)=0, \ \ \ \ \ \
\lambda H_{\lambda}(X, \overline{Y})= \langle X, \overline{Y}
\rangle |\nabla \lambda |^2 $$ for any type $(1,0)$ tangent vectors
$X$ and $Y$. In particular, one has $\lambda \Delta \lambda =
n|\nabla \lambda |^2$, and $\Delta e^{(n-1)u}=0$, where $\Delta
\lambda$ and $\nabla \lambda$ are the Laplacian and gradient of
$\lambda$ with respect to the Riemannian metric $g$. This completes
the proof of Theorem \ref{conformalchange}.
\end{proof}


As a consequence, since there is no non-constant positive harmonic
function on the Euclidean space, we know any Hermitian metric
conformal to the complex Euclidean metric $g_0$ on ${\mathbb C}^n$
cannot be G-K\"ahler-like unless it is a constant multiple of $g_0$.
The same is true for any G-K\"ahler-like manifold $(M^n,g)$ that is
complete and with nonnegative Ricci curvature, for exactly the same
reason.

\vsv

On the other hand, by Theorem \ref{conformalchange}, we could draw
the following conclusion:

\vsss

\noindent {\bf Example (G-K\"ahler-like metrics conformal to the
Euclidean metric).} Let $M^n\subset {\mathbb C}^n$ be an open subset
not equal to  ${\mathbb C}^n$. Let $g_0$ be the restriction on $M^n$
of the complex Euclidean metric. For any $p\in {\mathbb
C}^n\setminus M$, one can check directly that the metric
$\tilde{g}=\frac{1}{|z-p|^4}g_0$ on $M^n$ is G-K\"ahler-like.
Conversely, if $\tilde{g}=e^{2u}g_0$ is G-K\"ahler-like on $M^n$, then
by Theorem 4, we know that the function $\lambda = e^{-u}$ satisfies
$$ \frac{\partial^2\lambda }{\partial z_i \partial z_j} =0, \ \ \ \ \ \ \
\lambda \frac{\partial^2\lambda}{ \partial z_i \partial \overline{z}_j}
= 2\delta_{ij}\sum_{k=1}^n |\lambda_i|^2.$$
From this it follows that there must be a constant $c>0$ and a point
$p\in {\mathbb C}^n\setminus M$ such that $\lambda = c|z-p|^2$, hence
$e^{2u} = \frac{1}{c^2|z-p|^4}$.

\vsss

Our next goal is to introduce the right notion of bisectional
curvature and holomorphic sectional curvature. The novelty here is
only the definition of (Riemannian) bisectional curvature.

\vsv

We have two natural candidates for defining the Riemannian
bisectional curvature, namely, $R_{X\overline{X}Y\overline{Y}}$ and
$R_{X\overline{Y}Y\overline{X}}$. In the K\"ahler case, or more
generally the G-K\"ahler-like case, they are equal to each other,
and in general, their difference is \[ R_{XY\overline{XY}} =
R_{X\overline{X}Y\overline{Y}}- R_{X\overline{Y}Y\overline{X}}.\]
This gives us a one-parameter family of choices of Riemannian
bisectional curvature $B_a$ for any real number $a$:

\begin{definition}[\textbf{Bisectional curvatures}]

Given a Hermitian manifold $(M^n,g)$, and given any two non-zero
type $(1,0)$ tangent vectors $X$, $Y$ at $p$ in $M$,  the
(Hermitian) bisectional curvature $B^h(X,Y)$ and the Riemannian
bisectional curvature $B_a(X,Y)$ in the directions of $X$ and $Y$
are defined as \[ B^h(X,Y)=  \frac{
R^h_{X\overline{X}Y\overline{Y}} } {|X|^2|Y|^2}, \ \ \ \ B_a(X,Y)=
\frac{  a R_{X\overline{X} Y \overline{Y}}
 + (1-a) R_{X\overline{Y}Y\overline{X}} }  {|X|^2|Y|^2}.\]
The (Hermitian)
holomorphic sectional curvature and Riemannian holomorphic sectional
curvature in the direction of $X$ are defined by $H^h(X)=B^h(X,X)$ and
$H(X)=B_a(X,X)$, respectively.
\end{definition}


Note that $B_a(X,Y)$ and $B^h(X,Y)$ are both real valued, and $B_a(X,Y)=B_a(Y,X)$, but in general $B^h(X,Y)\neq B^h(Y,X)$. When the metric is K\"ahler-like, $B^h$ is symmetric, and when the metric is G-K\"ahler-like, $B_a$ is independent of $a$.

\vsv

The Riemannian bisectional curvature $B_a$ gives us a couple of
Ricci type curvature tensor: \[ Ric_a(X) = \sum_{i=1}^n B_a(X, e_i)
= a Ric_1(X) + (1-a) Ric_0(X) \] where $e$ is a unitary frame. Clear
they are independent of the choice of the unitary frame.

\begin{lemma}  \label{lemma 12}
On a Hermitian manifold $(M^n,g)$, if $X=\frac{1}{\sqrt{2}} (u-iJu)$
and $Y=\frac{1}{\sqrt{2}} (v-iJv)$, where $u$ and $v$ are real
tangent vectors, then we have
\begin{equation*}
 -R_{X\overline{X} Y \overline{Y}}  +2R_{X\overline{Y}Y\overline{X}} = - \frac{1}{2} \{ R(u,v)+R(Ju,Jv)+R(Ju,v)+R(u,Jv)\}
 \end{equation*}
where $R(u,v)$ stands for $R_{uvuv}$. Therefore
\begin{equation*}
 B_{-1}(X,Y) = \frac{1}{2}\sin^2\phi_{uv}\{ K_{u\wedge v} + K_{Ju\wedge Jv}\}   + \frac{1}{2}\sin^2\phi_{uJ\!v}\{ K_{Ju\wedge v} + K_{u\wedge Jv}\}
 \end{equation*}
where $K_{u\wedge v}=-R(u,v)/|u\wedge v|^2$ is the sectional
curvature of the plane spanned by $u$ and $v$, and $\phi_{uv}$
denotes the angle between $u$ and $v$. In particular, if $(M^n,g)$
has positive (negative, nonnegative, or nonpositive)  sectional
curvature, then it will have positive (negative, nonnegative, or
nonpositive) Riemannian bisectional curvature $B_{-1}$.

\end{lemma}

\begin{proof}
A straightforward computation leads to the above identities.
\end{proof}

In particular, we have
\begin{equation}
Ric_{-1}(X) = -Ric_0(X)+2Ric_1(X)  = \frac{1}{2} \{ Ric(u) + Ric(Ju)\}
\end{equation}
where $Ric(u)$ stands for the Ricci curvature in the direction of
$u$. This means that in the non-K\"ahler case, the trace of the
Riemannian bisectional curvature $B_{-1}$ is only the $J$-invariant part of
the Ricci curvature, which may not control the full Ricci curvature
tensor, even though the scalar curvature is controlled by it:
\begin{equation}
\sum_{i, j=1}^n B_{-1}(e_i,e_j) = \frac{1}{2} Scal
\end{equation}
where $\{ e_i\}$ is any unitary frame and $Scal$ stands for the scalar
curvature of the Riemannian metric $g$.

\vsss

Next we want to examine the relationship between $B_a(X,Y)$ and
$B^h(X,Y)$. As a direct consequence of the definitions and Lemma
\ref{lemma 7}, we get through a direct computation that the
following holds:

\begin{theorem} \label{monotonicity}
For any type $(1,0)$ tangent vectors $X$, $Y$ at a point $p$ in a
Hermitian manifold $(M^n,g)$, it holds
\begin{eqnarray}
\frac{1}{2}(R^h_{X\overline{X}Y\overline{Y}} +
R^h_{Y\overline{Y}X\overline{X}}) -
R_{X\overline{Y}Y\overline{X}}
& = & \sum_{k=1}^n \{ |T^k_{XY}|^2 + 2Re(T^Y_{kY} \overline{T^X_{kX}}) \}, \\
R^h_{X\overline{Y}Y\overline{X}} - R_{X\overline{X}Y\overline{Y}} &
= & \sum_{k=1}^n \{ |T_{kX}^Y|^2 + |T_{kY}^X|^2 - |T_{XY}^k|^2\},
\end{eqnarray}
where $\{e_i\}$ is a unitary frame and $T^X_{YZ} = \sum_{i,j,k=1}^n
T^i_{jk} \overline{X}_iY_jZ_k$ if $X=\sum_{i=1}^n X_ie_i$,
$Y=\sum_{i=1}^n Y_ie_i$, and $Z=\sum_{i=1}^n Z_ie_i$. In particular,
the holomorphic sectional curvature satisfies the monotonicity
condition
\begin{equation}
R^h_{X\overline{X}X\overline{X}} - R_{X\overline{X}X\overline{X}} \
= \ 2 \sum_{k=1}^n  |T^X_{kX}|^2  \ \geq \ 0.
\end{equation}
Moreover, if the equality always holds, then $T=0$ and $g$ is
K\"ahler.
\end{theorem}

Notice that if we write $x=\frac{1}{2} (X+\overline{X})$ and
$y=\frac{1}{2} (Y+\overline{Y})$, then  $\sum_{k=1}^n |T^X_{kY}|^2 =
2 ||(\nabla_x J)(y)||^2$. So the difference between the holomorphic
sectional curvatures is measured by the norm square of the covariant
differentiation of the almost complex structure. Note that $\nabla
J=0$ means that $g$ is K\"ahler.

\vsv

Formula (41) is particularly interesting. It says that the
difference between the symmetrized Hermitian bisectional curvature
and the Riemannian bisectional curvature $B_0$ is a quadratic
expression of the torsion tensor, and it does not involve the
derivatives of the torsion. Perhaps we should use $B_0$ to be the
Riemannian bisectional
 curvature, even though it is not clear to us whether $B_0$ can be
 expressed as a positive linear combination of sectional curvature
 terms as in the K\"ahler case.

\vsv

For Ricci curvature tensors, Liu and Yang wrote a nice paper recently
 \cite{Liu-Yang1} in which they systematically studied all 6 possible
  Ricci tensors, and wrote down their explicit relationship. So we
will not get into Ricci or scalar curvature here.

\vsv

To close this article, let us leave the readers with the following
vague question, namely, can we further study K\"ahler-like and
G-K\"ahler-like metrics on compact non-K\"ahlerian complex manifold
of dimension $3$ that is Calabi-Yau, that is, with trivial canonical
line bundle and finite fundamental group? Is there a role that
K\"ahler-like or G-K\"ahler-like metrics can play in the Strominger
system (\cite{Strominger}, \cite{Fu-Yau}) on such manifolds?

\vs

\noindent\textbf{Acknowledgments.} We would like to thank Bennett
Chow, Gabriel Khan, Kefeng Liu, Lei Ni, Hongwei Xu, and Xiaokui Yang
for their interests and encouragement. We also thank Qingsong Wang
who carefully read the manuscript and pointed out a number of
inaccuracies. The authors are grateful to an anonymous referee for
many helpful suggestions which improved the exposition of the paper.



\begin{thebibliography}{99}                     %


\bibitem {Boothby} W. Boothby, \emph{Hermitian manifolds with zero curvature.}
Michigan Math. J. {\bf 5}, (1958), no. 2, 229--233.


\bibitem {Calabi} E. Calabi,  \emph{Construction and properties of some 6-dimensional almost complex manifolds,}
Trans. Amer. Math. Soc. {\bf 87}, (1958), 407--438.


\bibitem {Calabi-Eckmann} E. Calabi and A. Eckmann,  \emph{A class of compact, complex manifolds which are not algebraic,} Ann. Math. (2)
 {\bf 58}, (1953), 494--500.

\bibitem {Fu} J.X. Fu,  {\em On non-K\"ahler Calabi-Yau threefolds with balanced metrics.} Proceedings of the International
Congress of Mathematicians. Volume II, 705-716, Hindustan Book Agency, New Delhi, 2010.



\bibitem {Fu-Yau} J.X. Fu and S.-T. Yau,  {\em The theory of superstring with flux on non-K\"ahler manifolds and the complex
Monge-Amp\`ere equation.} J. Differential Geom. {\bf 78} (2008), no. 3, 369--428.

\bibitem {Fu-Li-Yau} J.X. Fu, J. Li, and S.-T. Yau, {\em Constructing balanced metrics on some families of non-K\"ahler Calabi-Yau
threefolds.} J. Differential Geom. {\bf 90} (2012), no. 1, 81--129.


\bibitem {Fu-Wang-Wu} J.X. Fu, Z.Z. Wang, and D. Wu,  {\em Form-type Calabi-Yau equations.} Math. Res. Lett. {\bf 17} (2010), no.
5, 887--903.


\bibitem {Fu-Wang-Wu1} J.X. Fu, Z.Z. Wang and D. Wu,  \emph {Semilinear equations, the $\gamma_k$ function, and generalized Gauduchon
metrics.} J. Eur. Math. Soc. {\bf 15} (2013), no. 2, 659--680.


\bibitem {Gauduchon} P. Gauduchon, \emph{La {$1$}-forme de torsion d'une
vari\'et\'e hermitienne compacte.} Math. Ann. 267 (1984), no. 4,
495-518.

\bibitem {Goldstein} E. Goldstein and S. Prokushkin,  \emph{Geometric model for complex non-K\"ahler manifolds with SU(3) structure,}
Comm. Math. Phys. {\bf 251}, (2004), 65--78.

\bibitem {Gray1966} A. Gray, \emph{Some examples of almost Hermitian manifolds.}
Illinois J. Math. \textbf{10} (1966), 353--366.

\bibitem {Gray} A. Gray, \emph {Curvature identities for Hermitian and
almost Hermitian manifolds.} Tohoku Math. J. (2) {\bf 28} (1976),
no. 4, 601--612.

\bibitem {Guan-Sun} B. Guan and W. Sun,  \emph {On a class of fully
 nonlinear elliptic equations on Hermitian manifolds. }
 Calc. Var. Partial Differential Equations {\bf{54}} (2015), no. 1, 901-916.

\bibitem {Michelsohn} M.L. Michelsohn, \emph{On the existence of special metrics in complex geometry,} Acta Math.
 {\bf 149}, (1982), 261--295.

\bibitem {Li-Yau} J. Li and S.-T. Yau, \emph{The existence of supersymmetric
string theory with torsion.} J. Differential Geom. \textbf{70}
(2005), no. 1, 143--181.

\bibitem {Liu-Yang}  K.F. Liu and X.K. Yang, \emph
{Geometry of Hermitian manifolds,} Internat. J. Math. {\bf 23} (2012)

\bibitem {Liu-Yang1} K.F. Liu and X.K. Yang, \emph
{Ricci cuvratures on Hermitian manifolds,} arXiv: 1404.2481.


\bibitem {Liu-Yang2} K.F. Liu and X.K. Yang,
\emph {Hermitian harmonic maps and non-degenerate curvatures.} Math.
Res. Lett. 21 (2014), no. 4, 831-862.


\bibitem {Reid} M. Reid, \emph{The moduli space of {$3$}-folds with {$K=0$} may nevertheless
be irreducible.} Math. Ann. \textbf{278} (1987), no. 1-4, 329-334.


\bibitem {StreetsTian} Streets, J.; Tian, G..\emph{ A parabolic flow of pluriclosed
metrics.} Int. Math. Res. Not. IMRN 2010, no. 16, 3101-3133.

\bibitem {Strominger} A. Strominger, \emph{Superstrings with Torsion,}
Nuclear Phys. B {\bf 274}, (1986), 253--284.

\bibitem {Tosatti} V. Tosatti,  \emph { Non-K\"ahler Calabi-Yau manifolds,}
{Analysis, complex geometry, and mathematical physics: in honor of
{D}uong {H}. {P}hong}, 261-277, Contemp. Math., 644, Amer. Math.
Soc., Providence, RI, 2015.

\bibitem {Tseng-Yau} L.S. Tseng and S.-T. Yau, \emph{Non-K\"ahler Calabi-Yau
manifolds.} String-Math 2011, 241-254, Proc. Sympos. Pure Math.,
\textbf{85}, Amer. Math. Soc., Providence, RI, 2012.


\bibitem {Tosatti-Weinkove}    V. Tosatti and B. Weinkove, \emph { The complex Monge-Amp\`ere equation on compact Hermitian manifolds.} J.
Amer. Math. Soc. {\bf 23} (2010), no.4, 1187--1195.


\bibitem {Vaisman} I. Vaisman, \emph{Some curvature properties of complex surfaces.} Ann.
Mat. Pura Appl. (4) \textbf{132} (1982), 1-18 (1983).


\bibitem {Zheng} F. Zheng, {\em Complex differential geometry.} AMS/IP Studies in Advanced Mathematics, 18. American
Mathematical Society, Providence, RI; International Press, Boston,
MA, 2000.








\end{thebibliography}
\end{document}